\documentclass{amsart}
\usepackage{amsthm}
\usepackage{amsmath}
\usepackage{amssymb}
\begin{document}
\newtheorem{theo}{Theorem}
\newtheorem{cor}[theo]{Corollary}
\newtheorem{prop}[theo]{Proposition}
\newtheorem{lem}[theo]{Lemma}
\newcommand{\R}{\mbox{\bf R}}
\newcommand{\li}{\mbox{\rm li}\,}
\title{Sign changes of $\pi(x, q, 1) - \pi(x, q, a)$}
\author[J.-C. Schlage-Puchta]{Jan-Christoph Schlage-Puchta}

\begin{abstract}
It is known, that under the assumption of the generalized Riemannian
hypothesis, the function $\pi(x, q, 1) - \pi(x, q, a)$ has infinitely
many sign changes. In this article we give an upper bound for the
least such sign change. Similarly, assuming the Riemannian hypothesis
we give a lower bound for the number of sign changes of $\pi(x)-\li
x$. The implied results for the least sign change are weaker then
those obtained by numerical methods, however, our method makes no use
of computations of zeros of the $\zeta$-function.
\end{abstract}
\maketitle

\section{Introduction}

The
following question is known as the Shanks-Renyi-race problem:
Given an integer $q$, and a bijection $\sigma$ from the set $\{1, 2,
\ldots, \varphi(q)\}$ to the set of residue classes prime to $q$, is
it true that there are arbitrary large values $x$, such that the
inequalities 
\[
\pi(x, q, \sigma(1)) > \pi(x, q, \sigma(2)) > \dots > \pi(x, q,
\sigma(\varphi(q)))
\]
hold true?In this form the problem is unsolved for all $q$ with
$\varphi(q)>2$, even assuming the Generalized Riemannian
Hypothesis. With $\pi$ replaced by $\Psi$, it was solved by
J. Kaczorowski \cite{Kac2} for $q=5$, and the method develloped there
can be used for other small modules, too. However, the problem
involving $\pi$ is far more difficult, and the only result obtained so
far involving more then 2 residue classes was obtained by
J. Kaczorowski \cite{Kac1}, who showed that the function $\pi(x,
q, 1) - \max\limits_{a\not\equiv 1\pmod{q}}\pi(x, q, a)$ has
infinitely many sign changes. In \cite{JCP}, the same was shown by a
different method. In this note we use the method of \cite{JCP} to give
numerical bounds for the first sign change and for the number of sign
changes up to a given bound. We will prove the following theorem.

\begin{theo}
Let $q$ be a natural number, and set $q^+=\max(q, e(1260))$. Assume
that no $L$-series $\pmod{q}$ has zeros off the critical line. Let
$f(q)$ be the number of solutions of the congruence $x^2\equiv
1\pmod{q}$. Then there is an $x<e_2((q^+)^{170}+e^{18 f(q)})$ such
that $\pi(x, q, 1) > \pi(x, q, a)$ for all $a\not\equiv
1\pmod{q}$. Moreover, if $V(x)$ denotes the number of sign changes of $\pi(t,
q, 1) - \max\limits_{a\not\equiv 1\pmod{q}}\pi(t, q, a)$ in the range
$2\leq t\leq x$, we have
\[
V(x) > \frac{\log x}{\exp((q^+)^{170}+e^{18 f(q)})} - 1.
\]
\end{theo}
Here and in the sequel, $e_k(x)$ denotes the $k$-fold iterated
exponential function, and $\log_k x$ the $k$-fold iterated
logarithm. Note that the dependence on $f(q)$ is an immanent feature
of the problem, however, for almost all $q$ we have $f(q)<\log q$,
thus the least sign change is of order less then $e_3(55\log q\log_2
q)$ for almost all $q$.

By the same method bounds for sign changes of $\pi(x)-li\;x$ can be obtained.
Our result on the first sign change is substantially weaker than those given by
Skewes\cite{Ske}, Lehmann\cite{Leh} and te Riele\cite{tRe}, however, these
estimates involve large scale computation of zeros of Riemann's
$\zeta$-function 
and give no bound on the asymptotical behaviour of the number of sign changes.

\begin{theo}
Assume the Riemann hypothesis. Then there is an $x< e_3(16.7)$, such
that $\pi(x)> li\;x$. 
If $V(x)$ denotes the number of sign changes of $\pi(x) - li\;x$, we have
$V(x) > \frac{\log x}{e_2(16.7)} - 1$.
\end{theo}

A. E. Ingham\cite{Ing} proved that $V(x)>c\log x-1$ for some positive
constant $c$, however, his method of proof was ineffective. Without
the assumption of the Riemannian Hypothesis, slightly weaker estimates
were given by J. Pintz (see \cite{Pin1} for an ineffective, \cite{Pin2}
for an effective result). Moreover, J. Kaczorowski proved $V(x)>c\log
x-1$ unconditionally.

Since the proof of this theorem is easier, but shows all relevant details, we
will give this first.

Throughout this note, $\rho$ will denote nontrivial zeros of $\zeta$
or some $L$-series. Since we will always assume that all zeros are on
the critical line, we can write $\rho=\frac{1}{2}+i\gamma$ with
$\gamma$ real. For a real number $x$, $\|x\|$ denotes the distance of
$x$ to the nearest integer. Similar, for $x\in\R^n$, define $\|x\|$
to be the distance of $x$ to the nearest lattice point.

I would like to thank the anonymous referee for many helfull comments.

\section{Some Lemmata for Theorem 1}

We begin our computations with the following statement on the vertical
distribution of zeros of $\zeta$.

\begin{lem}
Denote with $N(T)$ the number of zeros $\rho$ of $\zeta$ with $0\Re\;\rho<1,
0<\Im\;\rho<T$. Then for $T>2$ we have $N(T) < \frac{1}{6}T\log T$ and
$N(T+1)-N(T) < \log T$.
\end{lem}
In fact Backlund\cite{Bac} gave a more precise estimate, however, this
lemma will suffice for our purpose. Even better estimates are
available under the Riemannian hypothesis, however, it seems difficult
to make these improvements explicit, and the bounds obtained that way
will not influence our final result significantly.

For this and the next section, define the functions
$\Delta(t)=\sum_\gamma\frac{e^{it\gamma}}{\rho}$ and
$\Delta_T(t)=\sum_{|\gamma|<T}\frac{e^{it\gamma}}{\rho}$, where both
summations run over roots of $\zeta$ on the critical line.

\begin{lem}
Let $a>b>0$ be real numbers with $a-b < \frac{1}{36}$ and $T>e^4$. Then we have
\begin{eqnarray*}
\int\limits_a^b |\Delta(t) - \Delta_T(t)|^2 dt & = & \sum_{|\gamma_1|,
  |\gamma_2|>T}\frac{1}{(1/2+i\gamma_1)(1/2+i\gamma_2)} 
\frac{e^{b(\gamma_1+\gamma_2)} - e^{a(\gamma_1+\gamma_2)}}{\gamma_1+\gamma_2}\\
 & < & \frac{2}{9}\frac{\log^3 T}{T}
\end{eqnarray*}
\end{lem}

If $\gamma_1 + \gamma_2 = 0$ then $\frac{e^{b(\gamma_1+\gamma_2)} -
e^{a(\gamma_1+\gamma_2)}}{\gamma_1+\gamma_2}$ 
denotes its limit for $\gamma_2\rightarrow-\gamma_1$, i.e. $b-a$. 

We will also need the following statement, which depends on a pigeon-hole
principle, for a proof see \cite{JCP}.

\begin{lem}
Let $n$ and $N$ be natural numbers, $\vec{\alpha} = (t_1, \ldots,
t_n)\in \R^n, \epsilon>0$. 
Then there is a sequence of $N$ real numbers 
$1<s_1 < \ldots < s_N<\frac{N2^n\Gamma(n/2)}{\pi^{n/2}\epsilon^n}+1 =:M+1$
such that for $1\leq i\leq N$ we have
\[
\|s_i\cdot(t_1,\ldots,t_n)\| < \epsilon
\]
and $s_{i+1} \geq s_i + 1$.
\end{lem}

Further we note that studying sign changes of $\pi$ is equivalent to
studying large values of $\Psi$, an observation which is made
exploicit by the following lemma.

\begin{lem}
Let $x>e^{60}$ be a real number such that
$\Psi(x)>x+1.01\sqrt{x}-2$. Then $\pi(x)>\li x$.
\end{lem}
\begin{proof}
The argument follows the lines of S. Lehmann\cite{Leh}. Define
$\Pi(x)=\sum_{n\leq x}\frac{\Lambda(n)}{\log n}$, and
$\Delta^*(x)=\Psi(x)-x$. We have an explicit formula
\[
\Pi(x)=\li x - \sum_\rho\li x^\rho + \theta x^{1/3},
\]
where $\theta$ is some real number, which depends on $x$ and satisfies
$|\theta|\leq 1$, provided that $x>e^{12}$. Further we have
\[
\li x^\rho = \frac{x^\rho}{\rho\log x} + \theta\frac{x^{1/2}}{|\rho|^2\log^2 x}
\]
where $\theta$ is some complex number satisfying $|\theta|\leq 1$.
Thus, comparing the sum over zeros with the sum occuring in the
explicit formula for $\Psi(x)$, we get
\[
\Pi(x)-\li x = \frac{\Psi(x)-x}{\log x} + \theta\left(\frac{1}{\log x}
\sum_\rho\frac{1}{|\rho|^2} + x^{-1/6}\log
x\right)\frac{\sqrt{x}}{\log x}.
\]
Finally, again under the assumption $x>e^{12}$, we have
\[
\pi(x)-\Pi(x)=-\frac{1}{2}\li \sqrt{x}+\theta x^{1/3}.
\]
Putting these
estimates together, we get
\[
\pi(x)-\li x = \frac{\Psi(x)-x}{\log x} - \frac{\sqrt{x}}{\log x}
\theta\left(\frac{1}{\log x} \sum_\rho\frac{1}{|\rho|^2} + x^{-1/6}\log
x\right)\frac{\sqrt{x}}{\log x}.
\]
Hence, under the assumptions $x>e^{12}$ and
$\Delta^*(x)>1.01\sqrt{x}-2$, we get
\[
\pi(x)-\li x \geq \frac{0.01\sqrt{x}}{\log x} -
\frac{0.05\sqrt{x}}{\log^2 x} - 2x^{1/3} -2,
\]
where we used the bound $\sum_\rho\frac{1}{|\rho|^2}<0.05$ (see the
proof of the next lemma). For $x>e^{60}$, the right-hand side of the
last equation becomes positive, and the proof of the lemma is
complete.
\end{proof}

Finally we need the following quantitative version of \cite{JCP}, Lemma 8.

\begin{lem}
We have
\[
|\Delta(t) + \Delta(-t)| < 0.0462
\]
\end{lem}
\begin{proof}
We have
\begin{eqnarray*}
|g(t) + g(-t)| & = & \frac{1}{2}\left|\sum_\rho \frac{e^{it\gamma} +
e^{-it\gamma}}{\rho} + \frac{e^{it\gamma} + e^{-it\gamma}}{\bar{\rho}}\right|\\
 & = & \frac{1}{2}\left|\sum_\rho \frac{e^{it\gamma} +
e^{-it\gamma}}{|\rho|^2}\right|\\
 & \leq & \sum_\rho\frac{1}{|\rho|^2}\\
 & = & 2 + C - \log\pi - 2\log 2\\
 & = & 0.04619\ldots
\end{eqnarray*}
Here $C = 0.5772\ldots\,$ denotes Euler's constant. The evaluation of
the sum $\sum_\rho\frac{1}{|\rho|^2}$ is given e.g. in \cite{Dav}.
Note that here we have twice the value given in \cite{Dav}, since we
take the sum over all zeros, not only zeros with positive imaginary part.
\end{proof}

\section{Proof of Theorem 2}

Obviously, the lower bound for $V(x)$ implies the bound for the first
sign change, hence, we will only consider the second claim of Theorem
2. Define $\Delta(t)$ and $\Delta_T(t)$ as above. We have for $t>0$
\[
\Psi(e^t) = e^t - e^{t/2}\Delta(t) - \frac{\zeta'}{\zeta}(0) -
\frac{1}{2}\log(1-e^{-2t}) 
\]
For $0<t<\log 2$ this becomes
\[
\Delta(t) = e^{t/2} - \left(\log 2\pi + \frac{1}{2}\log(1-e^{-2t})\right)
e^{-t/2} > 1 - \log 2\pi - \frac{1}{2}\log 2t 
\]
Together with Lemma 7 we obtain for $-\log 2 < t < 0$
\[
\Delta(t) < \frac{1}{2}\log (-t) + 1.25
\]
Especially we have $\Delta(t) < -1$ for $-e^{-4.6} < t < 0$. Now let $T>e^4$ be
a real number to be determined later, $M = N(T)$ the number of zeros of $\zeta$
with $0<\Im\,\rho\leq T$ and $\epsilon = \frac{1}{4\sqrt{M}}$. By
Lemma 5 there exists
a sequence of real numbers $s_i$, $1\leq i\leq N$ satisfying $s_1\geq
1$, $s_{i+1}\geq s_i+1$ and
\begin{equation}
s_N \leq
\frac{N(32\pi^2 M)^{M/2}\Gamma(M/2)}{\pi^{M/2}}<Ne^{\frac{3}{2}M\log M + 4M},
\end{equation}
such that 
\[
\left({\sum_\rho}^* |\arg s\gamma|\right)^2 \leq M{\sum_\rho}^* |\arg s\gamma|^2 \leq \frac{1}{2\sqrt{2}}
\]
where $\arg z$ is chosen to lie in the interval $[-\pi, \pi]$. Note
that with this choice we have $|\arg s\gamma  |\leq2\pi\|s\gamma\|$
For each such $s_i$ and every real $t$ we get
\begin{eqnarray*}
|\Delta_T(t) - \Delta_T(t+s_i)|^2 
 & \leq & \left(\sum_{|\gamma|\leq T} \left|\frac{e^{it\gamma}}{\rho} - \frac{e^{i(t+s_i)\gamma}}{\rho}\right|\right)^2\\
 & \leq & \left(\sum_{0\leq\gamma\leq T} \left|\frac{\arg s\gamma}{\rho}\right|\right)^2\\
 & \leq & \frac{1}{\gamma_0}\left(2\sum_{0\leq\gamma\leq T}|\arg s\gamma|\right)^2\\
 & \leq & \frac{1}{2\gamma_0}\\
 & = & \frac{1}{28.269\ldots}
\end{eqnarray*}
Now assume that $\Delta(t+s_i) > -1.01$ for all $t$ with $-e^{-4.6} < t < 0$.
Then on one hand we get
\begin{eqnarray*}
\int\limits_{e^{-4.6}}^0 |\Delta(t+s_i) - \Delta(t)|^2 dt & < &
\int\limits_{e^{-4.6}}^0 |\Delta(t) - \Delta_T(t)|^2 dt + 
\int\limits_{e^{-4.6}}^0 |\Delta_T(t+s_i) - \Delta_T(t)|^2 dt\\
&&\qquad + 
\int\limits_{e^{-4.6}}^0 |\Delta_T(t+s_i) - \Delta(t+s_i)|^2 dt \\
 & < & \frac{4}{9}\frac{\log^3 T}{T} + \frac{e^{-4.6}}{196}
\end{eqnarray*}
while on the other hand we have
\begin{eqnarray*}
\int\limits_{e^{-4.6}}^0 |\Delta(t+s_i) - \Delta(t)|^2 dt & > &
-\int\limits_0^{e^{-4.6}} (0.5\log t - 2.26)^2 dt\\
 & > & 0.32 e^{-4.6}
\end{eqnarray*}
These estimates contradict each other, provided that
$\frac{4}{9}\frac{\log^3 T}{T} < 0.31 e^{-4.6}$, i.e. for $T >
282000$. Thus we get $M < 590000$, and from (1) we conclude that $s_N
< N\cdot e_2(16.6)$. Now if $t>10$ then $\Delta(t)< -1.01$ implies
$\Psi(e^t) > e^t + 1.01e^{t/2} - 2$ and by Lemma 6 the latter implies
$\pi(e^t)>li\;e^t$, provided that $t>60$. Since there are at most 60
values $s_i$ excluded by the last condition, we see that in the
interval $[2, exp(N\cdot e_2(16.2))]$ there are at least $N-60$
values $x_i$, such that $x_{i+1} > e\cdot x_i$, and
$\pi(x_i)>li\;x_i$. Since 
\[
\int\limits_a^{e\cdot a} \frac{\Psi(e^t)-e^t}{e^{t/2}} dt <
\sum_\rho\frac{2}{|\gamma\rho|} < 0.1
\]
between $x_i$ and $x_{i+1}$ there is some $y_i$ such that $\pi(y_i) < li\;y_i$.
Hence in the interval $[2, exp(N\cdot e_2(16.2))]$ there are at least
$N-60$ sign changes of $\pi(x)-li\;x$. Our claim now follows from the
fact that $61\cdot e_2(16.6)<e_2(16.7)$.

\section{Lemmata for Theorem 1}

Fix a natural number $q>2$, and assume that no $L$-series $\pmod{q}$
vanishes in $\Re\;s > \frac{1}{2}$. In the sequel let $\chi$ be any
charakter $\pmod{q}$. We will prove Theorem 1 under the additional
assumption that $q>e(1260)$, it will be apparent from the proofs that
stronger conclusions than Theorem 1 can be obtained in the case of
small values of $q$, however, we do not believe that these results are
worth the additional effort.

Define the functions $\Delta(t,
\chi)=\sum_\gamma\frac{e^{it\gamma}}{\rho}$ and $\Delta_T(t,
\chi)=\sum_{|\gamma|<T}\frac{e^{it\gamma}}{\rho}$, where both
summations run over the nontrivial roots of $L(s, \chi)$.

\begin{lem}
Denote with $N(T, \chi)$ the number of zeros of $L(s, \chi)$ with
$0<\Re\;\rho<1, |\Im\;\rho|<T$. Then for $q, T > 10$ we have
\[
\left|N(T, \chi) - \frac{T}{\pi}\log\frac{qT}{2\pi} +
\frac{T}{\pi}\right| < \frac{1}{2.1}\log qT + 30 
\]
For $\log qT>1260$ and $q, T>40$ the bounds
\[
N(T, \chi) < \frac{1}{3} T\log qT
\]
and
\[
N(T+1, \chi) - N(T, \chi) < \log qT.
\]
Let $N_+(T, \chi)$ denote the number of zeros with $0\leq\gamma\leq
T$, and $N_-(T, \chi)$ the number of zeros with $0\geq\gamma\geq
-T$. Then we have for $q, T>40$ and $\log qT>1260$ the bound
\[
|N_+(T, \chi) - N_-(T, \chi)| <\frac{5}{4} \log qt.
\]
Finally, we have
\begin{equation}
\sum_\rho\frac{1}{|\rho|^2}\leq 13\log q.
\end{equation}
\end{lem}
\begin{proof}
The asymptotic bound for $N(T, \chi)$ follows from \cite[Theorem
2.1]{McK} setting $\eta=0.01$. The upper bound for $N(T, \chi)$
follows immediatelly from this estimate. For the upper bound for
$N(T+1, \chi) - N(T, \chi)$ we begin with the equation
\begin{equation}
\label{eq:Lapprox}
-\Re\frac{L'}{L}(s, \chi) = \frac{1}{2}\log\frac{q}{\pi} +
\frac{1}{2}\Re\frac{\Gamma'}{\Gamma}\left(\frac{s+a}{2}\right)-
\Re\sum_\rho\frac{1}{s-\rho},
\end{equation}
where the summation over the zeros has to be taken with respect to
increasing imaginary part, and $a=\frac{1-\chi(-1)}{2}$. To bound the
term coming from the $\Gamma$-function, we use the estimate (see\cite[6.1.42]{AS})
\[
\left|\log\Gamma(z) - \left(z-\frac{1}{2}\right)\log z + z - \frac{1}{2}\log
2\pi - \frac{1}{12 z}\right| \leq \frac{K(z)}{360|z^3|},
\]
where $K(z)=\sup_{u\in\R}\left|\frac{z^2}{z^2+u^2}\right|$, which for $\Re
z\in[5/4, 7/4]$ and $\Im z>40$ implies
\[
\left|\frac{\Gamma'}{\Gamma}(z)-\log z\right|\leq\frac{1}{79}, 
\]
which together with (\ref{eq:Lapprox}) implies
\[
-\Re\frac{L'}{L}(5/4+it, \chi)\leq \frac{1}{2}\log qt -
\Re\sum_\rho\frac{1}{s-\rho}-\frac{1}{2}.
\]
Set $t=T+1/2$, and assume that $N(T+1, \chi)-N(T,
\chi)>\frac{5}{4}\log qT$. Then 
every zero with imaginary part in the range $[T, T+1]$ would
contribute at least $\frac{12}{13}$ to the right-hand side sum, and the
last inequality would imply
\[
-\Re\frac{L'}{L}(5/4+it, \chi)\leq -\frac{2}{13}\log qT\leq -193,
\]
which would contradict the lower bound
\[
-\Re\frac{L'}{L}(5/4+it, \chi) \geq \frac{\zeta'}{\zeta}(5/4) \geq
\]
Finally, the bound comparing $N_+(T, \chi)$ and $N_-(T, \chi)$ can be
proven in the same way as \cite[Theorem 2.1]{McK}, and the bound
for $\sum_\rho\frac{1}{|\rho|^2}$ follows from the other estimates. 
\end{proof}

Just as in section 2 we get

\begin{lem}
Let $a>b>0$ be real numbers with $a-b < \frac{1}{324}$ and $T>e^4$. Set
$\Delta_T(t) = \sum_{|\gamma|>T}\frac{e^{it\gamma}}{1/2+i\gamma}$. Then we have
\[
\int\limits_a^b |\Delta_T(t)|^2 dt = \sum_{|\gamma_1|, |\gamma_2|>T}\frac{1}{(1/2+i\gamma_1)(1/2+i\gamma_2)}
\frac{e^{b(\gamma_1+\gamma_2)} - e^{a(\gamma_1+\gamma_2)}}{\gamma_1+\gamma_2}
< \frac{2}{9}\frac{\log^3 qT}{T}
\]
\end{lem}

For $x>1$ we have the explicit formula
\[
\Psi(x, \chi) = E_\chi x - \sqrt{x}\sum_\rho\frac{e^{i\gamma\log x}}{\rho} - d_\chi\log x - R(x, \chi) + B(\chi)
\]
where
\begin{eqnarray*}
E_\chi & = & \left\{\begin{array}{ll}1 & \mbox{ if }\chi = \chi_0\\ 0
    & \mbox{ otherwise}\end{array}\right.\\ 
d_\chi & = & \left\{\begin{array}{lcl}1 & \mbox{ if } & \chi(-1) = 1,
    \chi\neq\chi_0\\ 
         0 & \mbox{ if } & \chi(-1) = -1\mbox{ or }\chi=\chi_0
    \end{array}\right.\\ 
R(x, \chi) & = & \left\{\begin{array}{lcl}\frac{1}{2}\log(1-x^{-2}) &
    \mbox{ if } & \chi(-1) = 1\\ 
        \frac{1}{2}\log(1-x^{-2}) + \log\frac{x}{x+1}  & \mbox{ if } &
    \chi(-1) = -1\end{array}\right.\\ 
B(\chi) & = & -E_\chi + \log 2 - C + \log\frac{q}{\pi} +
    \frac{L'}{L}(1, \bar{\chi})\\ 
\end{eqnarray*}
The value of $B(\chi)$ can be obtained using the functional equation, see
\cite[Lemma 1]{Kac2}. 
Define $\Delta(t, q, a) := \frac{1}{\varphi(q)}\sum_\chi\chi(a)
\Delta(t, \chi)$. To estimate $\Delta(t, q, a)$ in a neighbourhood of
0, we need an upper bound for $B(\chi)$, and hence for $\frac{L'}{L}(1, \chi)$.

\begin{lem}
Let $q>10$ be an integer, and $\chi$ a character $\pmod{q}$. Then
there is some constant $\theta$ of absolute value at most 1, such that 
\[
\left|\sum_\chi \overline{\chi(a)}\frac{L'}{L}(1, \chi)\right| =
\frac{\varphi(q)\Lambda(a)}{a} + \vartheta\big(2\log^2 q +
9\sqrt{\varphi(q)\log q}\big).
\]
\end{lem}
\begin{proof}
The proof will be similar to the estimate given by Masley and
Montgomery\cite{MM}, however, things become easier since we assume GRH here.
Set $f(s) = \sum_\chi\overline{\chi(a)}\frac{L'}{L}(s, \chi)$, thus we have
to estimate $f(1)$. Using the
Brun-Titchmarsh inequality we get for $\sigma > 1$ the estimate
\[
\left|f(\sigma) - \frac{\Lambda(a)\varphi(q)}{a^\sigma}\right| < \frac{\Lambda(q+a)\varphi(q)}{(q+a)^\sigma}
 + 3 + \frac{1}{(\sigma-1)\log 2} + \log^2 q
\]
(see \cite{MM}, Lemma 1). Now differentiating the partial fraction
decomposition of $\frac{L'}{L}$ we get for $\sigma> 1$
\[
f'(\sigma) =
\sum_\chi\overline{\chi(a)}\sum_\rho\frac{1}{(\sigma-\rho)^2} +
\vartheta 
\]
where the inner sum runs over all nontrivial zeros of $L(s, \chi)$ and
$|\vartheta|<1$. We assume $\Re\;\rho=\frac{1}{2}$ for all $\rho$, so the
inner sum can be estimated using lemma 3 by $10\log q$, thus
$|f'(\sigma)| < 10\varphi(q)\log q + 1$. Finally
\[
\left|f(1) - \frac{\Lambda(a)\varphi(q)}{a^\sigma}\right| < 3 + \log^2 q + \log q + 
 \frac{1}{(\sigma-1)\log 2} + 10(\sigma-1)\varphi(q)\log q + (\sigma-1)
\]
Choosing $\sigma = 1 + \frac{1}{\sqrt{7\varphi(q)\log q}}$ we obtain
\[
\left|f(1) -  \frac{\Lambda(a)\varphi(q)}{a}\right| < 2\log^2 q +
8\sqrt{\varphi(q)\log q}, 
\]
which proves our claim.
\end{proof}
Now we have enough information to give an estimate for $\Delta(t, q,
a)$ for $t$ close to 0.

\begin{lem}
For $0<t<\log 2$, $q>e^{32}$ we have for some real $\theta$ satisfying
$|\theta|<1$ the estimate
\[
\Delta(t, q, 1) = \left(\log q - \frac{1}{2}\log(1-e^{-2t}) +
2\theta\right)e^{-t/2},
\]
and for $a\neq 1\pmod{q}$ we have the bound
\[
|\Delta(t, q, a)|\leq 3
\]
\end{lem}
\begin{proof}
We consider three cases: $a\equiv 1\pmod{q}$, $a\equiv -1\pmod{q}$ and
$a\not\equiv \pm1\pmod{q}$.

For $a\not\equiv 1\pmod{q}$, all contributions to $\Delta(t, \chi)$,
which are independent of $\chi$ cancel, if further $a\not\equiv
-1\pmod{q}$ terms depending only on $\chi(-1)$ cancel as well, 
so if $a\not\equiv\pm 1\pmod{q}$ we get for $0<t<\log 2$
\begin{eqnarray*}
\Delta(t, q, 1) & = & \frac{e^{t/2}-e^{-t/2}}{\varphi(q)} +
\frac{te^{-t/2}}{\varphi(q)} + \frac{e^{-t/2}}{\varphi(q)}
\sum_\chi\overline{\chi(a)}\frac{L'}{L}(1, \bar{\chi})\\ 
 & = & \frac{e^{t/2}-e^{-t/2}}{\varphi(q)} + \frac{te^{-t/2}}{\varphi(q)} +
 \frac{\Lambda(a)e^{-t/2}}{a} + \frac{\vartheta e^{-t/2}}{\varphi(q)}
\left(2\log^2 q + 8\sqrt{\varphi(q)\log q}\right).
\end{eqnarray*}
For $a\equiv -1\pmod{q}$ we get
\begin{eqnarray*}
\Delta(t, q, -1) & = & \frac{e^{t/2}}{\varphi(q)} +
\frac{(\varphi(q)/2-1)te^{-t/2}}{\varphi(q)} +
\frac{1}{2}\log\frac{e^t}{e^t+1} + \frac{e^{-t/2}}{\varphi(q)}
\sum_\chi\overline{\chi(a)}\frac{L'}{L}(1, \bar{\chi})\\ 
 & = & \frac{e^{t/2}}{\varphi(q)} + \frac{(\varphi(q)/2-1)te^{-t/2}}
{\varphi(q)} +\frac{1}{2}\log\frac{e^t}{e^t+1} + \frac{\vartheta e^{-t/2}}
{\varphi(q)}
\left(2\log^2 q + 8\sqrt{\varphi(q)\log q}\right).
\end{eqnarray*}
Finally for $a\equiv 1\pmod{q}$ we get
\begin{eqnarray*}
\Delta(t, q, 1) & = & \frac{e^{t/2}-e^{-t/2}}{\varphi(q)} +
\frac{(\varphi(q)/2-1)te^{-t/2}}{\varphi(q)} + \frac{e^{-t/2}}{\varphi(q)}
\sum_\chi\overline{\chi(a)}\frac{L'}{L}(1, \bar{\chi})\\
&&\qquad + e^{-t/2}\big(\log 2 - C + \log\frac{q}{\pi}-\frac{1}{2}
\log(1-e^{-2t})\big)\\ 
 & = & \frac{e^{t/2}-e^{-t/2}}{\varphi(q)} + \frac{(\varphi(q)/2-1)te^{-t/2}}
{\varphi(q)}+ e^{-t/2}\big(\log 2 - C + \log\frac{q}{\pi}-\frac{1}{2}
\log(1-e^{-2t})\big)\\ 
&&\qquad + \frac{\vartheta e^{-t/2}}
{\varphi(q)}
\left(2\log^2 q + 8\sqrt{\varphi(q)\log q}\right).
\end{eqnarray*}
For $q>6$ we have $\varphi(q)>\sqrt{q}$, using this together with the
bound $q>e^{32}$ we can conclude that the terms involving $\theta$ are
of absolute value $\leq 0.02$, and all the other terms with the
exception of $\frac{1}{2}\log(1-e^{-2t})$ and $\log q$ can easily be
bounded absolutely. Putting these bounds together, we obtain our
claim.
\end{proof}

\begin{lem}
We have for $|x|\leq 0.01$ and $q\geq\exp(1260)$ the bounds
\[
\left|\int_0^x\Delta(t, \chi)+\Delta(-t, \chi)\;dt\right| < 53x\log q
\]
and
\[
\left|\int_0^x\Delta(t, q, a)+\Delta(-t, q, a)\;dt\right| < 53x\log q.
\]
\end{lem}
\begin{proof}
It suffices to prove the first inequality, since the second is
obtained by averaging over all characters. Denote with $\rho_n$ the
$n$-th zero of $L(s, \chi)$ with positive imaginary part, $\rho_{-n}$
the $n$-th zero with negative imaginary part. By Lemma 8 we have
$|\gamma_n-\gamma_{-n}|<1$. Further we have
\[
|\Delta(t, \chi)+\Delta(-t, \chi)| = 
\left|\sum_\rho\frac{e^{t\gamma_n} + e^{-t\gamma_n} +
e^{t\gamma_{-n}} + e^{-t\gamma_{-n}}}{\rho}\right|, 
\]
and each single summand can be estimated as follows.
\begin{eqnarray*}
\frac{e^{it\gamma_n} + e^{-it\gamma_n}}{\rho_n} + \frac{e^{it\gamma_{-n}}
+ e^{-it\gamma_{-n}}}{\rho_{-n}}
& = & \frac{e^{it\gamma_n} +
e^{-it\gamma_n}}{|\rho|^2} - (e^{it\gamma_n} + e^{-it\gamma_n})
\left(\frac{1}{\overline{\rho_n}} - \frac{1}{\rho_{-n}}\right)\\
 && +\frac{1}{\rho_{-n}}\big((e^{it\gamma_{-n}} + e^{-it\gamma_{-n}})
 - (e^{it\gamma_n} + e^{-it\gamma_n})\big)\\
 & \leq & \frac{4}{|\rho_n|^2} + \frac{1}{\rho_n}\min(4, 2t),
\end{eqnarray*}
since
\begin{eqnarray*}
|e^{-it\gamma_{-n}})-e^{-it\gamma_n})| & = &
 |e^{-it\gamma_{-n}-it\gamma_n}-1|\\
 & < & \min(2, t\gamma_{-n} + t\gamma_n).
\end{eqnarray*}
We will use this estimate for small values of $\gamma_n$. For large
values of $\gamma_n$ we estimate the integral of a single term by
\[
\left|\int_0^x e^{it\gamma_n}\;dt\right| \leq \frac{2}{|\gamma_n|}.
\]
Putting these two estimates together and using (2), we obtain
\begin{eqnarray*}
\left|\int_0^x\Delta(t, \chi)+\Delta(-t, \chi)\;dt\right| & \leq &
\sum_n \min\left(\frac{4x}{|\rho_n|^2}+\frac{x^2}{|\rho_n|},
  \frac{4}{\gamma_n\rho_n|}\right)\\
& \leq & \sum_n \frac{4x}{|\rho_n|^2} + \sum_{\gamma_n<x^{-2}}
\frac{x^2}{|\rho_n|} + \sum_{\gamma_n\geq x^{-2}}
\frac{4}{\gamma_n|\rho_n|}\\
 & \leq & 52x\log q + \sum_{1\leq n\leq x^{-2}}
\frac{5x^2\log\big(q(n+1)\big)}{3n} + \sum_{n\geq
x^{-2}}\frac{5\log\big(q(n+1)\big)}{3n^2}\\
 & \leq & 52 x\log q + 2x^2(2\log(x^{-1}) + 1)(\log q + 2 \log(x^{-1})
 + 1)\\
 &&\quad +2x^2\log q + 2x^2\log(x^{-1})\\
 & \leq & 53 x\log q,
\end{eqnarray*}
provided that $x<0.01$ and $\log q>100$, hence our claim.
\end{proof}

The next lemma allows us to translate a statement on $\Psi(x, q,
1)-\Psi(x, q, a)$ into a statement on $\pi(x, q, 1)-\pi(x, q, a)$.
\begin{lem}
Let $q>\exp(1260)$ be an integer, $x>\exp(27q\log q)$ be a real number
such that $\Psi(x, q, 1)-\Psi(x, q,
a)>\frac{7f(q)}{\varphi(q)}\sqrt{x}$, where $f(q)$ is the number of
solutions of the congruence $x^2\equiv 1\pmod{q}$. Then we have
$\pi(x, q, 1)>\pi(x, q, a)$. On the other hand, is $a$ is a quadratic
nonresidue, and $\Psi(x, q, 1)<\Psi(x, q,
a)+\frac{\sqrt{x}}{\varphi(q)}$, we have $\pi(x, q, 1)<\pi(x, q, a)$.
\end{lem}
\begin{proof}
As in the proof of Lemma 6, we have for $x>e^{12}$ the relation
\[
\Pi(x, q, 1)-\Pi(x, q, a) = \frac{\Psi(x, q, 1)-\Psi(x, q, a)}{\log x}
+ \frac{2\theta\sqrt{x}}{\varphi(q)\log x}\left(\frac{1}{\log x}
\sum_\rho\frac{1}{|\rho|^2} + x^{-1/6}\log x\right)
\]
with some $\theta$ satisfying $|\theta|<1$. Using (2) we obtain
\[
\Pi(x, q, 1)-\Pi(x, q, a) \geq \frac{\Psi(x,q,1)-\Psi(x,q,a)}{\log x} 
- \frac{27\log q}{\log^2 x}\sqrt{x}.
\]
On the other hand, using the Brun-Titchmarsh inequality to estimate
the contribution of higher powers to $\Pi(x, q, 1)$, we obtain for
$x>q^8$ the estimate
\[
\Pi(x, q, 1)-\Pi(x, q, a) \leq \pi(x, q, 1)-\pi(x, q, a) +
\frac{6f(q)\sqrt{x}}{\varphi(q)\log q} + x^{1/3}.
\]
Putting these estimates together, we get for $q>\exp(1260)$ and
$x>\exp(27q\log q)$ the first estimate of our lemma. The proof of the
second estimate is similar, yet somewhat easier.
\end{proof}

\section{Proof of Theorem 1}
The proof begins as the proof of Theorem 2. By Lemma 11, we have for
$0<t<\log 2$
\[
\left|\Delta(t, q, 1) - e^{-t/2}(\log q -
\frac{1}{2}\log(1-e^{-2t}))\right| < 2,
\]
as well as
\[
|\Delta(t, q, a)|<3
\]
for $-1<t<1$, $(a, q)=1$, and $a\not\equiv 1\pmod{q}$. Applying Lemma
12, we obtain for $0< x\leq 0.01$ the bound
\[
\left|\int_{-x}^0\Delta(t, q, 1)\;dt - \int_0^xe^{-t/2}
\left(\frac{1}{2}\log(1-e^{-2t}) - \log q\right)\;dt\right| < 53x\log
q.
\]
Setting $x=q^{-120}e^{-15f(q)}$, we deduce that
\begin{equation}
\int_{-x}^0\Delta(t, q, 1)\;dt < \int_{-x}^0\min_{a\neq 1}\Delta(t, q,
1)\;dt - 4x\log q -7xf(q).
\end{equation}
From Lemma 9 we obtain that
\[
\int_{-x}^0\Delta_T(t, q, 1)\;dt < \int_{-x}^0\min_{a\neq 1}\Delta(t, q,
1)\;dt - 3x\log q -7xf(q),
\]
provided that 
\[
\frac{2\log^3 qT}{9T} < x\log q.
\]
The latter condition is satisfied for $T=q^{130}e^{16f(q)}$, provided
that $q>e^8$ ,since $f(q)<q$ holds trivially. From Lemma 8, the number
$M$ of zeros occuring in the sum for $\Delta_T(t)$ is at most $qT\log
qT\leq q^{140}e^{17f(q)}$. From Lemma 5, applied with $\varepsilon =
\frac{1}{4\pi^2 M}$ we obtain a sequence of real numbers $s_1, \ldots,
s_N$, such that $s_1\geq 1$, $s_{i+1}\geq s_i+q^3$, 
\begin{eqnarray*}
s_N\leq \frac{q^3N(8\pi^2M)^M}{\pi^{M/2}} & < &
\exp\left(\frac{3}{2}M\log M + 3M\right)\\
 & < & \exp\left(q^{150}e^{18f(q)}\right)
\end{eqnarray*}
and
\[
\left({\sum_{\rho}}^*|\arg s_i\gamma|\right)^2 \leq
M{\sum_\rho}^*|\arg s_i\gamma|^2 \leq 1,
\]
where summation runs over all nontrivial zeros of all $L$-series
$\pmod{q}$ with imaginary part $\gamma$ satisfying $|\gamma|\leq
q^{130}e^{16 f(q)}$. As in Section 3, this bound implies
\begin{equation}
|\Delta_T(t, q, a) - \Delta_T(t+s_i, q, a)|\leq 2 < \log q
\end{equation}
for all $(q, a)=1$ and $i=1, \ldots, N$. Now assume that for all
$t\in[-x, 0]$ we had 
\begin{equation}
\Delta(t+s_i, q, 1) > \min_{a\neq 1} \Delta(t+s_i, q, a) - \log q - 7
f(q).
\end{equation}
Then we get on one hand from (4) and Lemma 9 the estimate
\[
\int_{-x}^0 |\Delta(t+s_i, q, 1)-\Delta(t, q, 1)|\;dt < 2x\sqrt{\log
  q} + 2x,
\]
whereas on the other hand we have from (3) and Lemma 9, applied to
$\Delta(t, q, a)$ the bound
\[
\int_{-x}^0 |\Delta(t+s_i, q, 1)-\Delta(t, q, 1)|\;dt > 3x\log q -
2x\sqrt{\log q} - 2x,
\]
yielding a contradiction for $q>e^2$. Hence, for each $i$, there is
some $t\in[-x, 0]$, such that (5) fails for this value of $t$, and
from Lemma 13 we deduce that this implies
\[
\pi(e^{t+s_i}, q, 1)\geq \pi(e^{t+s_i}, q, a)
\]
for all $a\not\equiv 1\pmod{q}$, provided that $s_i>27q\log q$.

Repeating the same argument, this time starting with the inequality
\[
\int_{-x}^0 \Delta(t, q, 1)\;dt < \int_{-x}^0 \min_{a\neq 1}
\Delta(t, q, 1)\;dt - 4x\log q - 7xf(q)
\]
instead of (3), we find that for each $s_i$ there is some $t\in[s_i,
s_i+x]$ such that
\[
\pi(e^{t+s_i}, q, 1) \leq \pi(e^{t+s_i}, q, a)
\]
for all $a\not\equiv 1\pmod{q}$. Hence, there are at least $N-27q\log
q$ sign changes of $\pi(x, q, 1)-\max_{a\neq 1}\pi(x, q, a)$ below
$\exp(N\exp(q^{150}e^{18f(q)}))$, solving for $N$ yields the second
statement of Theorem 1, since
\[
27q\log q\exp(q^{150}e^{18f(q)}) \leq \exp(q^{160}e^{18f(q)}).
\]

\end{document}